\documentclass[review]{elsarticle}[13pt]

\usepackage{lineno,hyperref}
\modulolinenumbers[5]

\usepackage[margin=0.7in]{geometry}
\usepackage{amsmath,amsthm,verbatim,amssymb,amsfonts,amscd, graphicx, float}
\usepackage{enumitem, scalefnt}
\newtheorem{theorem}{Theorem}[section]
\newtheorem{corollary}{Corollary}[theorem]
\newtheorem{lemma}[theorem]{Lemma}
\newtheorem*{remark}{Remark}

\usepackage{bm}

\newcommand\Ix[0]{I_{x}}
\newcommand\Iy[0]{I_{y}}
\newcommand\Iz[0]{I_{z}}
\newcommand\Ixinv[0]{\Ix^{-1}}
\newcommand\Iyinv[0]{\Iy^{-1}}
\newcommand\Izinv[0]{\Iz^{-1}}
\newcommand\vl[0]{\bm{l}}
\newcommand\vL[0]{\bm{L}}
\newcommand\vn[0]{\bm{n}}

\newcommand\vlr[0]{\bm{l^r}}
\newcommand\vLr[0]{\bm{L^r}}
\newcommand\vnr[0]{\bm{n^r}}

\newcommand\vlb[0]{\bm{l^b}}
\newcommand\vLb[0]{\bm{L^b}}

\newcommand\vi[0]{\bm{i}}
\newcommand\vj[0]{\bm{j}}
\newcommand\vk[0]{\bm{k}}

\newcommand\veo[0]{{\bm{e_1}}}
\newcommand\vet[0]{{\bm{e_2}}}
\newcommand\vEo[0]{{\bm{\varepsilon_1}}}
\newcommand\vEt[0]{{\bm{\varepsilon_2}}}

\newcommand\lxb[0]{{\vlb_x}}
\newcommand\lyb[0]{{\vlb_y}}
\newcommand\lzb[0]{{\vlb_z}}

\newcommand\Lxb[0]{{\vLb_x}}
\newcommand\Lyb[0]{{\vLb_y}}
\newcommand\Lzb[0]{{\vLb_z}}

\newcommand\cosa[0]{\cos\alpha}
\newcommand\sina[0]{\sin\alpha}
\newcommand\cosb[0]{\cos\beta}
\newcommand\sinb[0]{\sin\beta}

\newcommand\psit[0]{{\psi_t}}
\newcommand\thetat[0]{{\theta_t}}
\newcommand\phit[0]{{\varphi_t}}

\newcommand\thetao[0]{{\theta_0}}
\newcommand\phio[0]{{\varphi_0}}
\newcommand\psio[0]{{\psi_0}}

\newcommand\cpsi[0]{\cos\psit}

\newcommand\ctheta[0]{\cos\thetat}
\newcommand\stheta[0]{\sin\thetat}
\newcommand\cphi[0]{\cos\phit}
\newcommand\sphi[0]{\sin\phit}

\newcommand{\norm}[1]{||#1||}
\newcommand{\inner}[2]{<#1, #2>}
\newcommand\modtpi[0]{\text{ mod } 2\pi}
\newcommand\modone[0]{\text{ mod } 1}
\newcommand{\normL}[0]{\norm{\vL}}
\newcommand{\E}[0]{\mathbb{E}}

\begin{document}

\begin{frontmatter}

\title{Dynamics and Probability in the Toss of a Coin \\with Symmetric Inhomogeneous Density}

\fntext[shilunfootnote]{Department of Mathematics, UC Berkeley, Berkeley, CA, 94709, USA. Email: \url{shilun@berkeley.edu}. Research supported by University of California, Berkeley under Berkeley Fellowship.}
\author{Shilun Li\thanks{}}




\begin{abstract}
Under investigation in this paper is the dynamics and probability of heads in the toss of a coin with symmetric inhomogeneous density. Such coins are assumed to have diagonal inertia matrix. The rotational motion of the coin is determined by the initial angular momentum and initial position of the coin. We described the dynamic behavior of the unit normal vector and calculated the limiting probability of heads as time goes to infinity with respect to the fixed initial parameters. Our probability formula extends the formula for homogeneous coins by Keller and Diaconis et al.
\end{abstract}

\begin{keyword}
coin toss, rigid body, limiting probability, dynamic equations
\end{keyword}

\end{frontmatter}


\section{Introduction}
The motion of a coin toss can be modeled with a dynamical system governed by mechanics laws, determined entirely on the initial configuration. The outcomes can be random due to the variations in the initial parameters. Several physical mechanisms for randomness in coin toss have been reported, see \cite{mahadevan2011probability}. 

Keller considered a specific uniform coin with initial velocity and angular velocity imparted at the instant of tossing \cite{keller1986probability}. The uniform coin has inertia matrix given by $\text{diag}(\Ix, \Iy, \Iz)$ with $\Ix=\Iy<\Iz$, spins without air resistance and lands without bouncing. Assuming that the coin rotates about a horizontal axis that lies along a diameter of the coin, Keller proved that the limiting probability of heads is $50\%$. Building upon Keller's work, Diaconis et al found dynamical bias in the toss of a uniform coin which depends on the angle $\theta$ between initial the angular momentum $\vL$ and normal of heads $\vn$ \cite{diaconis2007dynamical}. The probability of heads if 50\% if and only if $\theta=\frac{\pi}{2}$. If a coin starts out heads, it ends up heads more often. Diaconis et al also measured empirical distributions of $\theta$ from real coin flip experiments and estimated that the probability of heads is $50.83\%$ given the coin starts out heads.

While Keller and Diaconis et al neglects air resistance and bouncing of the coin, Vulovi{\'c} and Prange\cite{vulovic1986randomness} analysed the effect of bouncing on the probability. They found that bouncing adds randomness to the toss which results in an increase in fairness. Yue and Zhang\cite{zeng1985sensitive} takes into account both bouncing and air resistance. The non-linearity of air resistance and bouncing causes acute sensitivity to initial conditions, adding randomness to the coin toss. On the other hands, Lindley\cite{lindley1981coin} followed by Gelman et al\cite{gelman2002you} considered non-uniform coins with mass inhomogeneously distributed. They gave informal arguments without rigorous proofs suggesting that the inhomogeneity of the coin will not affect the probability if the coin is caught in hand.

In this paper, we will investigate the dynamical bias of coins with symmetric inhomogeneous density, which is also referred to as non-uniform coins. Non-uniform coins are coins with inertia matrix given by $\text{diag}(\Ix, \Iy, \Iz)$ where $\Ix<\Iy<\Iz$. We will neglect the influence of air resistance and bouncing, assuming that the coin rotates freely in the air.

\section{Preliminaries} 
\label{sec::2}
We will first introduce three coordinate systems centered at the centroid of the coin with orthonormal basis: 
\begin{itemize}
    \item Reference frame $\{\vi, \vj, \vk\}$ where $-\vk$ is the direction of gravity and $\vi,\vj$ independent of time.
    \item Body fixed frame $\{\veo, \vet, \vn\}$ where $\vn$ is the normal to the heads of the coin.
    \item Intermediate frame $\{\vEo,\vEt,\vl\}$ where $\vEo=\frac{\vk-\inner{\vk}{\vl}\vl}{\norm{\vk-\inner{\vk}{\vl}\vl}}$, $\vEt=\vn\times\vEo$, as shown in Figure \ref{coordinate}.
\end{itemize}
\begin{figure}[H]
\begin{center}
\includegraphics[width=5cm]{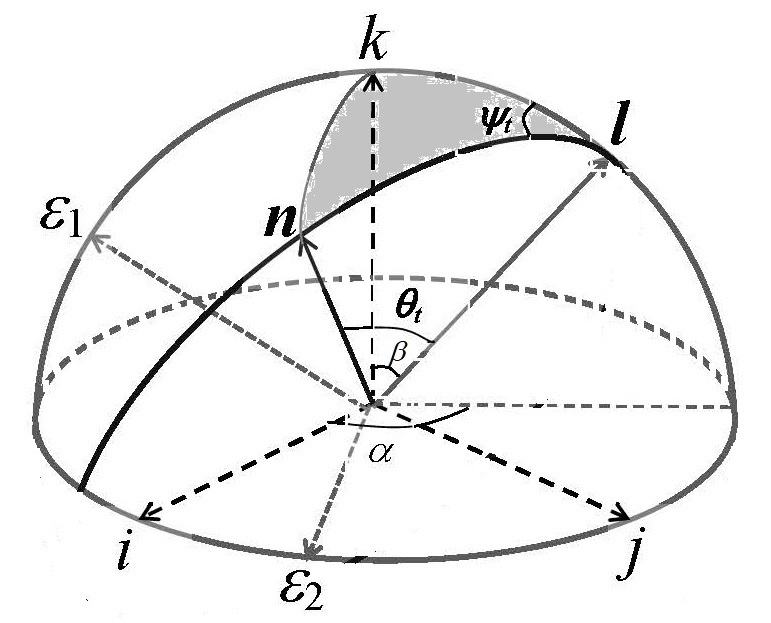}
\hspace{18pt}
\includegraphics[width=5cm]{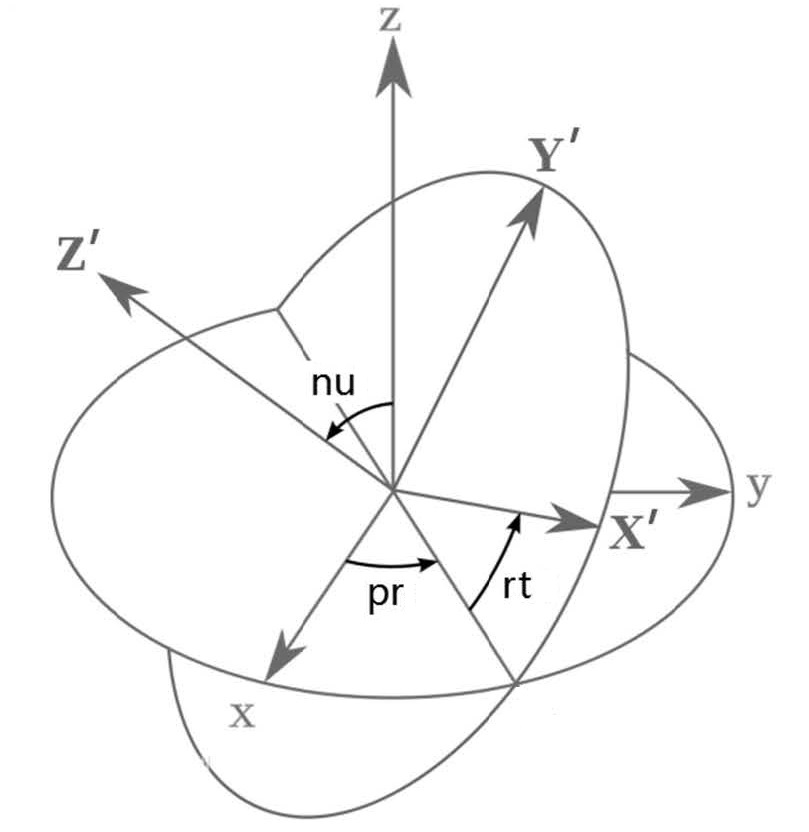}
\end{center}
\vspace*{8pt}
\caption{The basis set $\{\vEo,\vEt,\vl\}$ and the angles $\psit, \thetat,\alpha, \beta$ (left) and Euler angles (right).}
\label{coordinate}
\end{figure}

We use superscript r refer to the coordinates of vectors in the reference frame, b for the body fixed frame, and no superscript for basis independent situations. The intermediate frame is only introduced for calculating the rotational matrix between the body fixed frame and the reference frame and its existence will be suppressed in section \ref{sec::normal vec}. 

Angular momentum theorem applies in the reference frame \cite{murray1993probability,strzalko2008dynamics}. It tells us the angular momentum $\vLr$ is conserved in the reference frame since the coin is torque free if we ignore air resistance. Then the coordinates of $\vLr$ and $\vlr=\frac{\vLr}{\norm{\vLr}}$ are time independent in the reference frame. But the coordinates in body frame are time dependent. Using spherical coordinates, we can write 
\begin{equation}
\vlr = (\cosa\sinb, \sina\sinb, \cosb), 
\label{lr}
\end{equation}
and
\begin{equation}
\vlb=(\cphi\stheta, \sphi\stheta, \ctheta),
\label{lb}
\end{equation}
where $\alpha$, $\beta$ and $\thetat$ are shown in Figure \ref{coordinate}.

Any orthonormal basis can be rotated to another orthonormal basis by a sequence of three Euler angles \cite{goldstein2011classical}, precession pr, nutation nu, and rotation rt, as shown in the right panel of Figure \ref{coordinate}. The rotation matrix (acting by left multiplication) in terms of Euler angles is
\begin{equation}
    A=\begin{pmatrix}
    C_{pr}C_{rt}-S_{pr}C_{nu}S_{rt} & -C_{pr}S_{rt}-S_{pr}C_{nu}C_{rt} & S_{pr} S_{nu}\\
    S_{pr}C_{rt}+C_{pr}C_{nu}S_{rt} & -S_{pr}S_{rt}+C_{pr}C_{nu}C_{rt} & -C_{pr}S_{nu}\\
    S_{nu}S_{rt} & S_{nu}C_{rt} & C_{nu}\\
    \end{pmatrix},
    \label{Amatrix}
\end{equation}
where $S$ and $C$ denote the trigonometric functions $\sin$ and $\cos$, e.g. $S_{nu},\ C_{nu}$ denote $\sin(nu),\ \cos(nu)$, respectively.

Intermediate frame acts as a bridge between the body frame  and the reference frame. Let $A_1$ be the rotation matrix described by Euler angles $\{pr_1, nu_1, rt_1\}$ from $\{\vi,\vj,\vk\}$ to $\{\vEo,\vEt,\vl\}$, $A_2$ be the rotation matrix described by Euler angles $\{pr_2, nu_2, rt_2\}$ from $\{\vEo,\vEt,\vl\}$ to $\{\veo,\vet,\vn\}$. We have,
\begin{equation}
    (\vi,\vj,\vk)\xrightarrow[]{(pr_1, nu_1, rt_1)} (\vEo,\vEt,\vl) \xrightarrow[]{(pr_2, nu_2, rt_2)} (\veo,\vet,\vn)
\end{equation}
From the definition of the basis $\{\vEo,\vEt,\vl\}$,
 we have $\vk = (\sinb)\vEo + (\cosb)\vl$. So the coordinates of $\vk$ in the system $\{\vEo,\vEt,\vl\}$ and $\{\vi,\vj,\vk\}$ are $\big(\sinb,0,\cosb\big)$ and $(0,0,1)$, respectively. In addition, the coordinates of $\vl$ in the system $\{\vi,\vj,\vk\}$ and $\{\vEo,\vEt,\vl\}$ are $\big(\cosa\sinb,\sina\sinb,\cosb\big)$ and $(0,0,1)$, respectively. Then the rotation matrix $A_1$ satisfies
\begin{equation}
\begin{cases}
\left(\sinb,0,\cosb\right)^T = A_1^T (0,0,1)^T,\\
\left(\cosa\sinb,\sina\sinb,\cosb\right)^T = A_1 (0,0,1)^T.
\end{cases}
\end{equation}
The components of $A_1$ have form (\ref{Amatrix}). Those equations imply $(pr_1,\ nu_1, \ rt_1)=(\alpha + \frac{\pi}{2},\beta, \frac{\pi}{2})$ or $(\alpha - \frac{\pi}{2},-\beta,  -\frac{\pi}{2})$. Then we obtain
\begin{equation}
     \label{A1 matrix}
      A_1 = 
     \begin{pmatrix}
     -C_\alpha C_\beta & S_\alpha & C_\alpha S_\beta\\
     -S_\alpha C_\beta & -C_\alpha & S_\alpha S_\beta\\
     S_\beta & 0 & C_\beta\\
  \end{pmatrix}
\end{equation}

Let $\psit$ be the dihedral angle $\vk-\vl-\vn$, the dynamic angle of the plane spanned by $\vn(t)$ and $\vl$ rotating around the plane spanned by $\vk$ and $\vl$, as shown in Figure \ref{coordinate}. Then $\psit \modtpi$ is also the longitude of $\vn$ in the intermediate frame. Similarly, we have $(pr_2,\ nu_2, \ rt_2)= ( \psit + \frac{\pi}{2},\thetat, \frac{\pi}{2} - \phit)$ or $( \psit - \frac{\pi}{2},-\thetat, -\frac{\pi}{2} - \phit)$. Therefore
\begin{equation}
     \label{A2 matrix}
      A_2 = 
     \begin{pmatrix}
     -S_\psit S_\phit-C_\psit C_\thetat C_\phit &
     S_\psit C_\phit - C_\psit C_\thetat S_\phit &
     C_\psit S_\thetat
     \\
     C_\psit S_\phit - S_\psit C_\thetat C_\phit&
     -C_\psit C_\phit - S_\psit C_\thetat S_\phit&
     S_\psit S_\thetat
     \\
     S_\thetat C_\phit&
     S_\thetat S_\phit&
     C_\thetat
  \end{pmatrix}
\end{equation}

\section{The evolution of normal vector \texorpdfstring{$\vnr$}{TEXT}}
\label{sec::normal vec}
\subsection{Dynamic equations of angular momentum \texorpdfstring{$\vlb$}{TEXT}}
The coin rotates freely not subject to any net forces or torques around the fixed centroid. This is a classical Euler-Poinsot problem. The dynamic equations are given in Landau\cite{landau1969mechanics} by
\begin{equation}
\begin{cases}
    \frac{d}{dt}\Lxb = (\Izinv - \Iyinv)\Lxb \Lzb,\\
    \frac{d}{dt}\Lyb = (\Ixinv - \Izinv)\Lzb \Lxb,\\
    \frac{d}{dt}\Lzb = (\Iyinv - \Ixinv)\Lxb \Lyb.
\end{cases}
\label{dynamic_equation_L}
\end{equation}
Or in terms of Euler angles, 
\begin{equation}
\begin{cases}
\frac{d\psit}{dt}= \normL\left(\frac{\cos^2\phit}{\Ix}+\frac{\sin^2\phit}{\Iy}\right),\\
\frac{d\phit}{dt}=\normL \cos \thetat\left(\frac{\cos^2\phit}{\Ix}+\frac{\sin^2\phit}{\Iy}-\frac{1}{\Iz}\right),\\
\frac{d\thetat}{dt}=\frac{\normL}{2}\left(\Ixinv-\Iyinv \right)\stheta\sin(2\phit).\\
\end{cases}
\label{dynamic_euqation_angles}
\end{equation}
Note that for non-uniform coins, there is no explicit analytical solution for $\vLb(t)$. The rotational kinetic energy of the coin is given by
\begin{equation}
    E
    =\frac{1}{2}\bigg(\frac{(\Lxb)^2}{\Ix}+\frac{(\Lyb)^2}{\Iy}+\frac{(\Lzb)^2}{\Iz}\bigg)
    = \frac{\normL^2}{2}\bigg(\frac{(\lxb)^2}{\Ix}+\frac{(\lyb)^2}{\Iy}+\frac{(\lzb)^2}{\Iz}\bigg)
    \label{energy_sphere}
\end{equation}
which is constant with respect to $t$. Therefore, $\vlb$ must lie on the fixed ellipsoid $\frac{x^2}{\Ix}+\frac{y^2}{\Iy}+\frac{z^2}{\Iz}=\frac{2E}{\normL}$ and the sphere $x^2+y^2+z^2=1$ in the body fixed frame for all $t$. The intersection is a closed curve as shown in Figure \ref{angular momentum traj}. So $\vlb$ is periodic. In the special case of uniform coins, the angular velocity or angular momentum rotates and traces out a circle in the body-fixed frame.
\begin{figure}[H]
\begin{center}
\includegraphics[width=5cm]{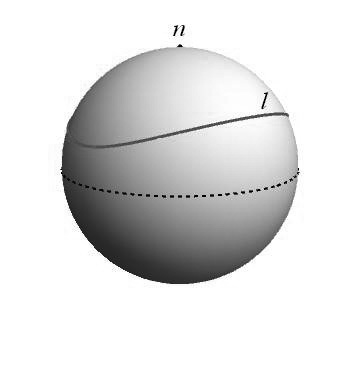}
\hspace{6pt}
\includegraphics[width=5cm]{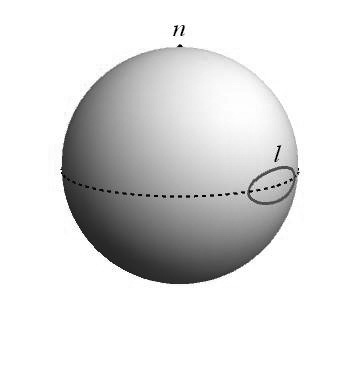}
\end{center}
\vspace*{6pt}
\caption{Two possible paths of $\vlb$ for non-uniform coins.}
\label{angular momentum traj}
\end{figure}

\subsection{The normal \texorpdfstring{$\vnr$}{TEXT}} 
Based on the evolution of $\vLb$ in the body frame, and the motion of $\vLb$ relative to normal vector $\vn$, we can further derive the evolution of $\vnr$ in reference frame.

\begin{theorem}
\label{normal evolution}
Given an initial angular momentum
$$\vLr = \normL\left(\cosa\sinb, \sina\sinb, \cosb \right).$$
Then at time $t$, the unit normal vector
\begin{align*}
    \vnr=&
    \begin{pmatrix}
    -C_{\alpha}C_\beta C_\psit S_\thetat+S_{\alpha} S_\psit S_\thetat + C_{\alpha}S_\beta C_\thetat\\
    -S_{\alpha}C_\beta C_\psit S_\thetat
    -C_{\alpha} S_\psit S_\thetat
    +S_{\alpha} S_{\beta} C_\thetat\\
    S_{\beta} C_\psit S_\thetat
    +C_{\beta} C_\thetat\\
    \end{pmatrix},
\end{align*}
where $(\phit, \thetat, \psit)$ are determined by equations (\ref{dynamic_euqation_angles}). 
\end{theorem}
\begin{proof}
The theorem directly follows from equations (\ref{A1 matrix}), (\ref{A2 matrix}) and $\vnr = A_1A_2(0,0,1)^T$.
\end{proof}
In Figure \ref{normal traj}, the coin is heads up when $\vnr$ is at the north hemisphere and tails is up otherwise. The figures show that $\vn$ precesses around the angular momentum $\vl$. For uniform coins, $\vn$ spin around $\vl$ in a circle, and the angle $\thetat$ between $\vl$ and $\vn$ stay constant. For non-uniform coins, $\vn$ spin with nutation around $\vl$ in a ring between 2 parallel circles.
\begin{figure}[H]
\begin{center}
\includegraphics[width=5cm]{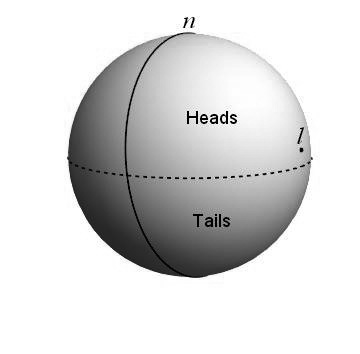}
\hspace{6pt}
\includegraphics[width=5cm]{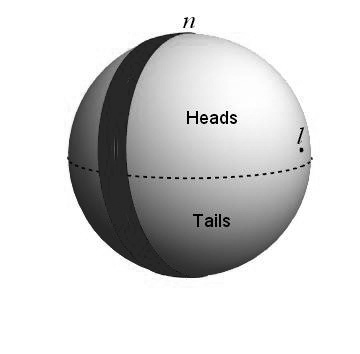}
\end{center}
\vspace*{6pt}
\caption{The path of $\vnr$ for uniform (left) and non-uniform (right) coin.}
\label{normal traj}
\end{figure}
From Theorem \ref{normal evolution}, we obtain the criterion for the coin landing heads up:
\begin{corollary} $\vnr(t)$ satisfies
\begin{equation}
    \vnr(t)\cdot\vk = \cosb\ctheta + \sinb\stheta\cpsi,
\label{criteria_eqation}
\end{equation}
and the coin is head up at time $t$ if and only if 
$$\sinb  \stheta\cpsi>-\cosb \ctheta.$$
\label{col::criteria}
\end{corollary}
Equation (\ref{criteria_eqation}) is just the law of cosines for the spherical triangle, the shaded part in the left panel in Figure \ref{coordinate}, which is formed by the endpoints of  unit vectors $\vn,\vk$ and $\vl$, 

\begin{remark}{\textbf{(property on precession $\psit$)}}
\label{remark::precession}
For the uniform coins with $\Ix=\Iy$, $\vn$ precesses around $\vl$ at a constant speed $\normL\Ixinv$. However, for non-uniform coins, the $\vn$ precesses around $\vl$ at speed varying from $\normL\Iyinv$ to $\normL\Ixinv$.
\end{remark}

\begin{remark}{\textbf{(property on nutation $\thetat$)}}
\label{remark::nutation}
For uniform coins, $\thetat$ is constant. For non-uniform coins $\thetat$ varies periodically from $\theta_m$ to $\theta_M$, which is given by:
\begin{align} 
    \theta_m&=\arccos(\sqrt{c_2}),\ \theta_M=\pi-\arccos(\sqrt{c_2}), &\text{ if } c_1<0,\\
    \theta_m&=\arccos(\sqrt{c_2}),\ \theta_M=\arccos(\sqrt{c_1}), &\text{ if } c_2\geq 0 \text{ and } \thetao\in[0,\frac{\pi}{2}],\\
    \theta_m&=\pi-\arccos(\sqrt{c_1}),\ \theta_M=\pi-\arccos(\sqrt{c_2}), &\text{ if }c_1\geq 0\text{ and } \thetao\in(\frac{\pi}{2},\pi].
\end{align}
with $c_1,c_2$ given by (\ref{c1}), (\ref{c2}).
\end{remark}
\begin{proof}
Since $\norm{\vlb}=1$ and $\vlb$ satisfies equation (\ref{energy_sphere}). Adopting Lagrange method, the extremums of $\lzb$, locating on $\lxb=0$ or $\lyb=0$, are given by
\begin{equation}
    c_1=\cos^2{\theta_0}-\frac{\Ixinv-\Iyinv}{\Iyinv-\Izinv}\cos^2{\phio}\sin^2{\thetao}
    \label{c1}
\end{equation}
\begin{equation}
    c_2=\cos^2{\thetao}+\frac{\Ixinv-\Iyinv}{\Ixinv-\Izinv}\sin^2{\phio}\sin^2{\thetao}
    \label{c2}
\end{equation}
Since $\lzb=\ctheta$, For the uniform case, $c_1=c_2$ so $\thetat$ is constant. For the non-uniform case, we have the desired result by considering the three situations: $c_1<0$, $c_2\geq 0$ and $\thetao\in[0,\frac{\pi}{2}]$, $c_1\geq 0$ and $\thetao\in(\frac{\pi}{2},\pi]$.
\end{proof}
$\theta_m$ and $\theta_M$ are are independent of $\normL$. They are determined by $\Ix, \Iy$ and $\Iz$. For non-uniform coins, $\theta_m$ and $\theta_M$ are either supplementary, both acute, or both obtuse. When the initial $\thetao$ is close enough to $\frac{\pi}{2}$, contained in 
\begin{equation}
S_F=\{(\phio, \thetao)|\cot^2\thetao
<\frac{\Ixinv-\Iyinv}{\Iyinv-\Izinv}\cos^2{\phio}\}, 
\label{SF}
\end{equation}
then $\theta_m$ and $\theta_M$ are supplementary. Under this condition, the two boundary circles perpendicular to $\vl$ corresponding to $\theta_m$ and $\theta_M$ in Figure \ref{normal traj} are centered symmetrically around the spherical center. Let us denote the fair region as the set of initial parameters $(\phio, \thetao)$ such that the proportions of "heads" zone and "tails" zone of $\vnr$ are equal (to 50\%). $S_F$ in (\ref{SF}) is the fair region for non-uniform coins. On the other hand, for uniform coins, the proportion of "heads" zone is 50\% if and only if $\thetao=\frac{\pi}{2}$. The fair region is shown in Figure \ref{fair region}.
\begin{figure}[H]
\begin{center}
\includegraphics[width=7cm]{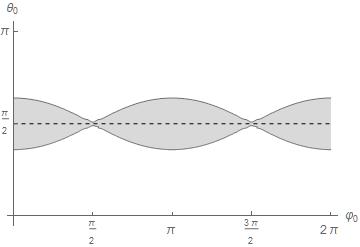}
\end{center}
\caption{Fair region of initial parameter $(\phio, \thetao)$\\
Dotted line: Uniform coin\   \   $ \ $Gray region: non-uniform coin}
\label{fair region}
\end{figure}
The probability of heads (which we will formulate in section \ref{sec::probability}) is approximately the proportion of the "heads" zone of $\vnr$. So we can assume the coin is fair when the initial parameters $(\phio,\thetao)$ is inside the fair region.

\section{Probability of Heads}
\label{sec::probability}
As the orientation of the coin is determined by $\normL$, $(\phio,\thetao)$ and $t$, we will define the probability of heads $p$ as the limiting probability of $\vnr_z>0$ as $t\rightarrow\infty$ given $(\phio,\thetao)$ and a distribution on $\normL$. When $(\phio, \thetao)$ is in the fair region, referred to as the fair case, we assume $p$ is 50\%. So let us now consider the situation where $(\phio, \thetao)$ is outside the fair region. Notice that in this situation, either $\thetat<\frac{\pi}{2}$ for all $t$ as show in the left of Figure \ref{angular momentum traj}, which we refer to as the acute case, or $\thetat>\frac{\pi}{2}$ for all $t$, which we refer to as the obtuse case. The key to obtaining $p$ is to obtain the limiting joint distribution of $(\psit,\phit)$. However, it is suffice to obtain the limiting distribution of $(\psit\modtpi,\phit\modtpi)$. This will rely on the following lemma about limit distributions. 
\begin{lemma}
\label{lem::limit lemma}
If $X$ is a random variable with characteristic function vanishing at infinity. Let $g_1,...,g_n:\mathbb{R}\rightarrow\mathbb{R}$ be real valued continuous function such that $\lim_{t\rightarrow \infty}|\sum_{i=1}^n m_i g_i(t)|=\infty$ for any $(m_1,..,m_n)\in \mathbb{Z}^n\setminus \{0\}$. Then $(g_1(t)X\mod 1,...,g_n(t)X\mod 1)$ converges in distribution to $\mathcal{U}[0,1]^n$ as $t\rightarrow\infty$. 
\end{lemma}
\begin{proof}

Let us consider the case where $n=2$, cases where $n>2$ follows similarly. Let $Y_1=(g_1(t)X\modone)$ and $Y_2=(g_2(t)X\modone)$. The cumulative distribution function and characteristic function of $X$ are denoted by $F(x)$ and $\Phi_X$, respectively. The characteristic function  $\Phi_Y$ of $(Y_1,Y_2)$  is determined by $\Phi_X$. Consider $\Phi_Y(2\pi m_1,2\pi m_2)$ for any $(m_1,m_2)\in \mathbb{Z}^2$, which are called Fourier coefficients in Engel \cite{engel1992road}, we have  
\begin{align*}
  &\Phi_Y(2\pi m_1,2\pi m_2)\\
  =&\E\left[\exp\{i(2\pi m_1Y_1+2\pi m_2Y_2\}\right]\\
=&\int_{-\infty}^\infty \exp\{2\pi i( m_1Y_1+ m_2Y_2)\} dF(x)\\
=&\sum_{k_1,k_2\in\mathbb{Z}}
\int_{I_{k_1,k_2}} \exp\{2\pi i( m_1Y_1+ m_2Y_2)\} dF(x) \\
=&\sum_{k_1,k_2\in\mathbb{Z}}
\int_{I_{k_1,k_2}} \exp\{2\pi i[ m_1(g_1(t)X-k_1)+ m_2(g_2(t)X-k_2)]\} dF(x) \\
=&\sum_{k_1,k_2\in\mathbb{Z}}
\int_{I_{k_1,k_2}} \exp\{2\pi i(m_1g_1(t)X+ m_2g_2(t)X)\} dF(x) \\
=&\int_{-\infty}^\infty \exp\{2\pi i( m_1g_1(t)X+ m_2g_2(t)X)\} dF(x)\\
=&\Phi_X[m_1g_1(t)+m_2g_2(t)]
\end{align*}
where $I_{k_1,k_2}=[k_1,k_1+1)\times [k_2,k_2+1)$. When $t\rightarrow\infty$, $m_1g_1(t)+m_2g_2(t)\rightarrow\infty$ and $\Phi_X[m_1g_1(t)+m_2g_2(t)]\rightarrow 0$. So 
\begin{equation}
    \lim_{t\rightarrow\infty} \Phi_Y(2\pi m_1,2\pi m_2) = 0
\end{equation}
for all $(m_1,m_2)\in\mathbb{Z}^2\setminus\{(0,0)\}$. Since $Y$ is supported by $[0,1]\times[0,1]$, and the Fourier coefficients of $\mathcal{U}[0,1]^2$ are zero, according to page 361 of Billingsley\cite{billingsley1968probability},
\begin{equation}
    Y \xrightarrow[]{\mathcal{D}} \mathcal{U}[0,1]^2
\end{equation}
which completes the proof.
\end{proof}

\begin{lemma}
\label{lem::limit distribution phi psi}
Suppose $\thetat<\frac{\pi}{2}$ for all $t$ or $\thetat>\frac{\pi}{2}$ for all $t$. For all Schwartz densities of $\normL$, all initial parameters $(\phio,\thetao,\Ix,\Iy,\Iz)$ excluding a measure 0 set, when $t\rightarrow \infty$, we have $(\psit\modtpi,\phit\modtpi)\xrightarrow[]{\mathcal{D}} \mathcal{U}[0,2\pi]^2$ in distribution.
\label{joint uniform distribution}
\end{lemma}
\begin{proof}
Recall that $(\lxb,\lyb,\lzb)$ lies on a closed curve and is periodic with some period $T$ and $(\phit, \psit)$ is its spherical coordinates. So $(\phit, \psit)$ has period $T$. Let us denote $c_1$ as the maximum of $\lzb$ on the curve and $c_2$ as the minimum of $\lzb$ on the curve. Define
\begin{equation}
\label{h}
    h(t):=\int_{0}^t \cos\theta_\tau
    \left(\frac{\cos^2\varphi_\tau}{\Ix}+\frac{\sin^2\varphi_\tau}{\Iy}-\frac{1}{\Iz}\right)d\tau,
\end{equation}
and
\begin{equation}
\label{g}
    g(t):=\int_0^t \left(\frac{\cos^2\varphi_\tau}{\Ix}+\frac{\sin^2\varphi_\tau}{\Iy}\right)d\tau.
\end{equation}

From (\ref{dynamic_euqation_angles}), we have
\begin{equation}
    \phit = \phio + \normL h(t),\quad
    \psit = \psio + \normL g(t)
\end{equation}
Now since $\normL$ is Schwartz, the characteristic function vanishes at infinity. By Lemma \ref{lem::limit lemma}, it is suffice for us to show
\begin{equation}
    \lim_{t\rightarrow\infty}|m_1 h(t)+m_2 g(t)|=\infty,\quad \forall m_1,m_2\in\frac{1}{2\pi}\mathbb{Z}^2\setminus\{0\}
\end{equation}
which is equivalent to the condition
\begin{equation}
    m_1 h(T)+m_2 g(T)\neq 0,\quad \forall m_1,m_2\in\frac{1}{2\pi}\mathbb{Z}^2\setminus\{0\}
    \label{condition not 0}
\end{equation}
since $h'(t)$ and $g'(t)$ is periodic with period $T$. By (\ref{energy_sphere}), we get the equality:
\begin{equation}
      \frac{\sin^2\thetat}{\sin^2\thetao} 
      = \frac{\frac{2E}{\normL^2}-\Izinv}{\Ixinv-\Izinv-(\Ixinv-\Iyinv)\sin^2\phit}\\
    = \frac{\Ixinv-\Izinv-(\Ixinv-\Iyinv)\sin^2\phio}{\Ixinv-\Izinv-(\Ixinv-\Iyinv)\sin^2\phit}.
    \label{sin theta_t}
\end{equation} 
Using equation (\ref{sin theta_t}), we can rewrite $m_1h'(t)+m_2g'(t)$ as
\begin{equation}
    m_1h'(t)+m_2g'(t) = \frac{a_0 (m_1\ctheta+m_2)}{\sin^2\thetat}+\frac{m_2}{\Iz}=\frac{a_0 (m_1\lzb+m_2)}{1-(\lzb)^2}+\frac{m_2}{\Iz}
\end{equation}
where $a_0=\sin^2\thetao\left[\Ixinv-\Izinv-(\Ixinv-\Iyinv)\sin^2\phio\right]$. To make a change of variable, we need $\frac{dz(t)}{dt}$. From (\ref{dynamic_euqation_angles}) and the relation
\begin{equation}
    (\lxb)^2=\frac{a_0-\Iyinv+(\Iyinv-\Izinv)(\lzb)^2}{\Ixinv-\Iyinv},\quad (\lyb)^2=\frac{-a_0+\Ixinv-(\Ixinv-\Izinv)(\lzb)^2}{\Ixinv-\Iyinv}
\end{equation}
given by (\ref{energy_sphere}), we have
\begin{equation}
    \frac{d\lzb}{dt}=\pm\normL\sqrt{A(\lzb)^4+B(\lzb)^2+C}
\end{equation}
for 
\begin{align}
    A=& -(\Ixinv-\Izinv)(\Iyinv-\Izinv),\\
    B=& \frac{2\Iz-\Ix-\Iy}{\Ix\Iy\Iz}-a_0(\Ixinv+\Iyinv-2\Izinv),\\
    C=& -(a_0-\Ixinv)(a_0-\Iyinv),
\end{align}
where $+$ is taken when $\lzb$ is going from $c_1$ to $c_2$ and $-$ is taken when $\lzb$ is going from $c_2$ to $c_1$. So with a change of variables from $t$ to $\lzb$, we can finally express the condition in (\ref{condition not 0}) as
\begin{equation}
    m_1h(T)+m_2g(T)=\frac{2}{\normL}\int_{c_1}^{c_2}\frac{\frac{a_0 (m_1z+m_2)}{1-z^2}+\frac{m_2}{\Iz}}{\sqrt{Az^4+Bz^2+C}}dz\neq 0,\quad\forall (m_1,m_2)\in\frac{1}{2\pi}\mathbb{Z}^2\setminus\{0\}
\end{equation}
It is very rare that one of the countable the integral equations 
\begin{equation}
    \int_{c_1}^{c_2}\frac{\frac{a_0 (m_1z+m_2)}{1-z^2}+\frac{m_2}{\Iz}}{\sqrt{Az^4+Bz^2+C}}dz = 0
    \label{integral equation}
\end{equation}
has a solution. We will assume that the set of $(\phio,\thetao,\Ix,\Iy,\Iz)$ such that (\ref{integral equation}) has a solution for some $(m_1,m_2)\in\frac{1}{2\pi}\mathbb{Z}^2\setminus\{0\}$ has Lebesgue measure 0 in $\mathbb{R}^5$. The proof will be left as an open problem. So for all $(\phio,\thetao,\Ix,\Iy,\Iz)$ excluding a measure 0 set, the condition to Lemma \ref{lem::limit lemma} is satisfied. According to the lemma, 
\begin{equation}
    (\psit\modtpi,\phit\modtpi)\xrightarrow[]{\mathcal{D}} \mathcal{U}[0,2\pi]^2
\end{equation}
as $t\rightarrow\infty$ which completes the proof.
\end{proof}

Let us assume from now on that the parameters $(\phio,\thetao,\Ix,\Iy,\Iz)$ are not in the measure 0 set of Lemma \ref{lem::limit distribution theta}. Now it remains for us to find the distribution of $\thetat$ before we can calculate the probability of heads. The following lemma obtains the distribution of $\thetat$ via its relation with $\phit$:
\begin{lemma}
\label{lem::limit distribution theta}
Suppose $\thetat<\frac{\pi}{2}$ for all $t$ or $\thetat>\frac{\pi}{2}$ for all $t$. $\thetat$ is a function of $\cos (2\phit)$ independent of $(\psit\modtpi)$. When $t\rightarrow\infty$, we have $\csc^2\thetat\xrightarrow[]{\mathcal{D}}\text{Arcsine}(a, b)$ for some $a,b$ where $\text{Arcsine}(a, b)$ denotes the Arcsine distribution on $(a,b)$. Moreover, the limiting pdf of $\thetat$ as $t\rightarrow\infty$ is given by
\begin{equation}
    f_\theta(y)=\frac{2|\cot y|}{\pi\sqrt{|(1-\csc^2\theta_m\sin^2 y)(\csc^2\theta_M\sin^2 y-1)|}}.
\end{equation}
\end{lemma}
\begin{proof}
By (\ref{sin theta_t}) we get
\begin{equation}
    \sin^2\thetat =\frac{\sin^2\thetao\left[\Ixinv-\Izinv-(\Ixinv-\Iyinv)\sin^2\phio\right]}{\Ixinv-\Izinv-(\Ixinv-\Iyinv)\sin^2\phit}
    =\frac{1}{k_1\cos(2\phit)+k_2}
\end{equation}
for some $k_1,k_2$ determined by $\Ix, \Iy, \Iz$,$\phio$ and $\thetao$. 
By Lemma \ref{lem::limit distribution phi psi}, we know that $\phit\xrightarrow[]{\mathcal{D}} \mathcal{U}[0,2\pi]$ as $t\rightarrow\infty$. So we have
\begin{equation}
    \csc^2\thetat\xrightarrow[]{\mathcal{D}} \text{Arcsin}(k_2-k_1, k_2+k_1)
\end{equation}
as $t\rightarrow\infty$. Comparing the expression of $k_1,k_2$ with the formula of $\theta_m, \theta_M$, we conclude that in the acute case:
\begin{equation}
    \csc^2\thetat\xrightarrow[]{\mathcal{D}}\text{Arcsine}(\csc^2\theta_M,\csc^2\theta_m),
\end{equation}
and in the obtuse case:
\begin{equation}
    \csc^2\thetat\xrightarrow[]{\mathcal{D}}\text{Arcsine}(\csc^2\theta_m,\csc^2\theta_M),
\end{equation}
as $t\rightarrow\infty$. So in both acute and obtuse cases, the limiting pdf of $\csc^2\thetat$ is given by:
\begin{equation}
    f(x)=\frac{1}{\pi\sqrt{|(x-\csc^2\theta_m)(\csc^2\theta_M-x)|}}.
\end{equation}
So the limiting pdf of $\thetat$ is given by:
\begin{equation}
\label{pdf of thetat}
    f_\theta(y)=\frac{2|\cot y|}{\pi\sqrt{|(1-\csc^2\theta_M\sin^2 y)(\csc^2\theta_m\sin^2 y-1)|}}.
\end{equation}
which completes the proof.

\end{proof}
Now with the result of Lemma \ref{lem::limit distribution phi psi} and Lemma \ref{lem::limit distribution theta}, we can start calculating the probability of heads.
\begin{theorem}
\label{thm::probability heads}
For all Schwartz densities of $\normL$, the limiting
probability of heads as $t\rightarrow\infty$ with $(\phio, \thetao)$ fixed, is given by
\begin{equation}
    \label{probability formula}
    p(\beta,\phio,\thetao)=\frac{1}{2}+\frac{1}{\pi^2}\int_{\theta_m(\phio,\thetao)}^{\theta_M(\phio,\thetao)}\arcsin(\min\{1, \cot\beta\cot y\})f_\theta(y) dy,
\end{equation}
with $f_\theta(y)$ is given in Lemma \ref{lem::limit distribution theta}. In the special case when $\Ix=\Iy<\Iz$(uniform coins), 
\begin{equation}
\label{probability formula uniform}
    p(\beta,\phio,\thetao)=p(\beta,\thetao)=\frac{1}{2}+\frac{1}{\pi}\arcsin(\min\{1, \cot\beta\cot \thetao\}).
\end{equation}

\end{theorem}
\begin{proof}
Let us first consider the case where $\thetat<\frac{\pi}{2}$ for all $t$ or $\thetat>\frac{\pi}{2}$ for all $t$.
By Lemma \ref{lem::limit distribution theta}, the limiting pdf of $\thetat$ is $f_\theta(y)$, given in \ref{pdf of thetat}.By Lemma \ref{lem::limit distribution phi psi}, the limiting pdf of $\psit\modtpi$ is $f_\psi(x)=\frac{1}{2\pi}$.
Let us define 
\begin{align*}
    D=\{x\in(0,2\pi),y\in(\theta_m,\theta_M)|\cos(x)>-\cot(\beta)\cot(y)\},\\
    D_1=\{x\in(0,\pi),y\in(\theta_m,\theta_M)|\cos(x)>-\cot(\beta)\cot(y)\}.
\end{align*}
By Corollary \ref{col::criteria}, $D$ is the region where the coin is heads up. Thus the limiting probability of heads is given by
\begin{align*}
    p(\beta,\phio,\thetao)
    =&\iint_D f_\psi(x)f_\theta(y) dx dy\\
    =& 2\iint_{D_1} f_\psi(x)f_\theta(y) dx dy\\
    =&\frac{1}{\pi}\iint_{D_1} f_\theta(y) dx dy\\
    =&\frac{1}{\pi}\int_{\theta_m}^{\theta_M} \arccos(-\min\{1, \cot\beta\cot y\})f_\theta(y)dy\\
    =&\frac{1}{2}+\frac{1}{\pi}\int_{\theta_m}^{\theta_M}\arcsin(\min\{1, \cot\beta\cot y\})f_\theta(y)dy
\end{align*}
as desired. When $(\phio,\thetao)$ lies in the fair region, $\theta_m$ and $\theta_M$ are complementary and $p=\frac{1}{2}$. Notice that (\ref{probability formula}) still holds as 
\begin{equation}
    \int_{\theta_m}^{\theta_M}\arcsin(\min\{1, \cot\beta\cot y\})f_\theta(y)dy
    =0
    \label{integral is 0}
\end{equation}
since
\begin{equation}
    \arcsin(\min\{1, \cot\beta\cot y\})f_\theta(y)=-\arcsin(\min\{1, \cot\beta\cot (\pi-y)\})f_\theta(\pi-y).
    \label{odd function equation}
\end{equation}
So (\ref{probability formula}) holds for all situations of $\thetat$. In the special case when $\Ix=\Iy<\Iz$, $\theta_m=\theta_M$ and the integral is integrated at a direct delta distribution at $\thetao$, which gives us 
\begin{equation}
    p(\beta,\phio,\thetao)=p(\beta,\thetao)=\frac{1}{2}+\frac{1}{\pi}\arcsin(\min\{1, \cot\beta\cot \thetao\})
\end{equation}
as desired.
\end{proof}
We can also see from Equation (\ref{odd function equation}) that it is natural to assume the probability is 50\% in the fair case. Note that $\vlb$ traces out a curve symmetric along the x-y plane and as shown on the right of Figure \ref{angular momentum traj}. If the limiting distribution of $\vlb_z$ as $t\rightarrow\infty$ exists, then since $\vlb_z(t\mod T)=-\vlb_z(-t\mod T)$, the limiting distribution of $\vlb_z$ will be symmetric along the x-y plane. Thus by Equation (\ref{odd function equation}) and $\thetat=\arccos(\vlb_z)$, the integral in (\ref{integral is 0}) is 0, giving us a probability of heads of 50\%. 

Usually, coin flips tend to start with $\beta=\thetao$, i.e. face of the coin facing straight up. An immediate corollary to Theorem \ref{thm::probability heads} is
\begin{corollary}
\label{col::probability straight}
With the assumptions of Theorem \ref{thm::probability heads} and further assuming that heads is facing straight up at the initial position, the limiting probability of heads as $t\rightarrow\infty$ with $(\phio,\thetao)$ fixed, is given by
\begin{equation}
    \label{probability formula straight}
    p(\phio,\thetao)=\frac{1}{2}+\frac{1}{\pi^2}\int_{\theta_m(\phio,\thetao)}^{\theta_M(\phio,\thetao)}\arcsin(\min\{1, \cot\thetao\cot y\})f_\theta(y) dy,
\end{equation}
In the special case when $\Ix=\Iy<\Iz$(uniform coins), 
\begin{equation}
\label{probability formula uniform straight}
    p(\phio,\thetao)=p(\thetao)=\frac{1}{2}+\frac{1}{\pi}\arcsin(\min\{1, \cot^2\thetao\}).
\end{equation}
\end{corollary}
Formula (\ref{probability formula uniform}) for uniform coins is also shown in Theorem 2 by Diaconis et al.\cite{diaconis2007dynamical} Furthermore, if the flip is Keller flip ($\thetao=\frac{\pi}{2}$), then $p$ is just $\frac{1}{2}$. 

We can assume that a normal coin toss starts with heads facing straight up. So it remains for us to find the distribution of the initial parameters $(\phio,\thetao)$. We use the $\thetao$ values from the 27 real flip experiments by Diaconis et al\cite{diaconis2007dynamical} to be the empirical distribution of $\thetao$. And we will assume that $\phio$ is uniform distributed in $[0,2\pi)$. Then using the probability formula in Corollary \ref{col::probability straight}, we can calculate the probability of heads of a normal non-uniform coin toss with the assistance of the computer. Like Diaconis et al\cite{diaconis2007dynamical}, we use an American half dollar which has $\Ix=6.68g\cdot cm^2$, $\Iz=13.24g\cdot cm^2$ and assume that $\Iy=7.35g\cdot cm^2$. As a result, we obtain the probability of heads of a non-uniform coin $P=50.45\%$. This is closer to $50\%$ compared to the probability of the uniform coin calculated by Diaconis et al \cite{diaconis2007dynamical} which is $50.83\%$. This shows that non-uniform coins are fairer than uniform coins.

\section{Conclusions}
While Coin-tossing is often used to make a decision between two options, the tossed coins are usually not absolutely uniform in our daily life. In this work, we investigated the dynamic behavior of non-uniform coins whose inertia matrix is given by $\text{diag}(\Ix, \Iy, \Iz)$ where $\Ix<\Iy<\Iz$. These coins include homogeneous coins with axis symmetrical convex parts, such as ellipse, rectangular, oblong shapes, on the surface and symmetrical inhomogeneous coins. 

We expressed the status, heads or tails, in terms of the initial direction of the angular momentum, the precession and nutation of the normal vector. We provided calculation of the limiting probability of heads as $t\rightarrow\infty$, with fixed initial direction of the angular momentum and distribution on magnitude of the angular momentum. The results from Keller\cite{keller1986probability} and Diaconis\cite{diaconis2007dynamical} are special cases of our study. 

In Figure \ref{fair region}, the fair region of initial parameters $(\varphi_0,\theta_0)$ of non-uniform coin has positive area while the fair region of the uniform coins is only a line in $\mathbb{R}^2$. The area of the fair region for non-uniform coin depends on $\Ix,\Iy,\Iz$. So there are much more situations of initial conditions where the non-uniform coin is fair and the uniform in not fair. 
\begin{figure}[H]
\begin{center}
\includegraphics[width=5cm]{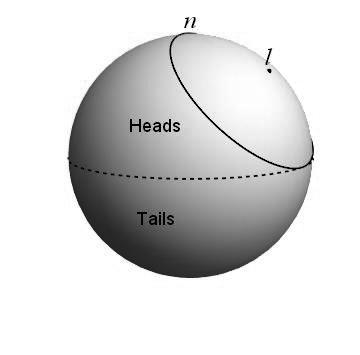}
\hspace{6pt}
\includegraphics[width=5cm]{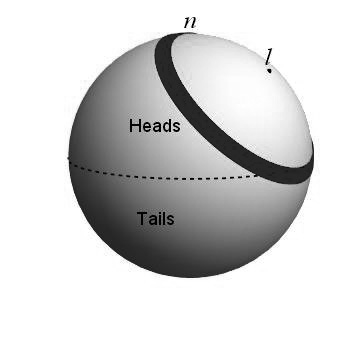}
\end{center}
\vspace*{6pt}
\caption{The path of $\vn$ for uniform (left) and non-uniform (right) coin when $\beta=\theta_0=\pi/4$.}
\label{compare_traj}
\end{figure}
In addition, Equation (\ref{probability formula uniform}) implies the probability of heads for uniform coin is 100\% if $\theta_0 \leq \frac{\pi}{4}$. Figure \ref{compare_traj} shows a situation ($\beta=\theta_0=\pi/4$) where the non-uniform coin is clearly fairer than the uniform coin. The two figures are possible regions of the normal vector with the same angular momentum for the uniform coin, non-uniform coin, respectively. Note that in left panel the possible region of unit vector of the uniform coin is inside the northern hemisphere. The coin never turns over and therefore the probability of heads is 100\%. But note that in the right panel, there is a small region inside the southern hemisphere due to nutation. Intuitively, we see there should be a small probability of the coin landing in tails. Corollary \ref{col::probability straight} proves that in this situation, the probability of heads is strictly less than 100\% for non-uniform coins.\\
\\
\section*{Acknowledgements} 
I would like to express my deepest thanks to Prof. Persi Diaconis from Stanford University for his guidance and mentorship on this research. His lectures on Mathematics and Statistics of Gambling greatly inspired me and gave me valuable insights on this topic.

\bibliography{mybibfile}

\end{document}